\documentclass{amsart}
\usepackage{bbm}
\usepackage{bbding}
\usepackage[mathcal]{euscript}
\usepackage{graphicx}
\usepackage{amsthm}
\usepackage{calc}
\usepackage{mathrsfs,dsfont}
\usepackage{CJK,fancyhdr,amscd}
\usepackage{anysize} 
\usepackage{amsmath,amssymb,amsfonts}
\usepackage{epsfig,enumerate}
\usepackage{latexsym}
\usepackage{indentfirst, latexsym}
\usepackage{graphics}
\usepackage[all,poly,knot]{xy}
\usepackage[colorlinks,linkcolor=blue,anchorcolor=blue,citecolor=blue,pagebackref]{hyperref}
\usepackage{geometry}
\usepackage{amsmath,amssymb,amsthm,mathtools}

\numberwithin{equation}{section}
\newtheorem{theorem} {Theorem} [section]
\newtheorem{proposition}[theorem]{Proposition}
\newtheorem{corollary}  [theorem]     {Corollary}
\newtheorem{lemma}  [theorem]     {Lemma}

\theoremstyle{definition}

\renewcommand*\backref[1]{}
\renewcommand*\backrefalt[4]{ \ifcase #1 \or (cited on page #2) \else (cited on pages #2) \fi}

\DeclareMathOperator{\sech}{sech}

\newcommand{\dd}{\,\mathrm{d}}
\newcommand{\inner}[2]{\left\langle #1,#2\right\rangle}
\newcommand{\II}{\mathrm{II}}
\begin{document}
\title{On questions of Pogorelov and Toponogov} 
\author{Lei Ni}
\address{Lei Ni. School of Mathematical Sciences, Zhejiang Normal University, Jinhua, Zhejiang 321004, China  \  and \ Department of Mathematics, University of California, San Diego, La Jolla, CA 92093, USA}
\email{leni@zjnu.edu.cn; \ lni.math.ucsd@gmail.com}

\author{Wei Zhang} 
\address{Wei Zhang. School of Mathematical Sciences, Zhejiang Normal University, Jinhua, Zhejiang 321004, China }
\email{39900230@zjnu.edu.cn}

\author{Yijian Zhang}
\address{Yijian Zhang. School of Mathematical Sciences,  University of Science and Technology of China, Hefei, Anhui, China, 230026}
\email{zyj\_math@mail.ustc.edu.cn}
\date{}

\markleft{Ni, Zhang and Zhang}
\markright{Questions of Pogorelov and Toponogov}

\begin{abstract}
We give  explicit counterexamples to two questions. One is asked by Pogorelov and the other is by Toponogov, at least both are stated in \cite{Topo}. These questions concern the existence of closed asymptotic curves in a saddle surface, namely a complete immersed regular surface in $\mathbb{R}^3$ with nonpositive Gaussian/sectional curvature, and its geometric consequences under some topological conditions. We also modify the statements and prove a corrected version. In the appendix we include an example clarifying a conjecture of Milnor.
\end{abstract}

\subjclass[2020]{53A05 (primary), 53B25 (secondary)}
\keywords{Saddle surfaces;  Asymptotic curves; Gauss curvature; Milnor conjecture}

\maketitle

\tableofcontents

\section{Introduction}

This  paper concerns two questions on closed asymptotic curves in a saddle surface. Recall that a complete surface $\Sigma$ in $\mathbb{R}^3$ is called a saddle surface if its Gauss curvature $K\le 0$. By classical geometry of surfaces \cite{DoCarmo, Pogorelov} the assumption  $K\le 0$ (at a point) is necessary to ensure the  existence of an asymptotic direction, namely the tangent vector $V$ which annihilates the second fundamental form $\II(\cdot, \cdot)$.  A regular curve whose tangent vector is asymptotic everywhere, or equivalently whose normal curvature vanishes everywhere, is called an asymptotic curve.
The study of asymptotic curves is crucial in the proof of Hilbert's theorem asserting that there is no isometric immersion of $\mathbb{H}^2$ in $\mathbb{R}^3$ \cite{Hopf, DoCarmo}, as well as in Efimov's celebrated generalization. They are also equally useful in the recent progress \cite{Moore} on the impossibility of isometric immersion of $\mathbb{H}^n$ into $\mathbb{R}^{2n-1}$. 

  For surfaces, Pogorelov (Problem 2.7.5 on page 125  of \cite{Topo}) formulated the following statement on the geometric consequence of the existence of a closed asymptotic curve  together with some assumption on global topology and geometry:

{\it (Pogorelov) If on a saddle surface which is homeomorphic to a plane, there exists a closed asymptotic curve, then the region bounded by this curve must be a region on the plane, namely isometric to a region in a plane.}

In the mean time, Toponogov formulated a different statement (Problem 2.7.4 on page 125 of \cite{Topo}) for surfaces with $K<0$:

{\it (Toponogov) On a regular surface with negative Gauss curvature there are no closed asymptotic curves.}

We shall present counterexamples to both statements. In the mean time we shall prove that the second claim holds if one assumes that the surface is homeomorphic to a plane.

\begin{proposition}\label{prop:1}
On a regular surface with negative Gauss curvature, which is homeomorphic to a plane there are no closed asymptotic curves.
\end{proposition}

 The counterexample to the first statement in fact shows that in the above result one can not relax the negativity of the Gauss curvature to the nonpositivity even with $K=0$ on some curves.

Notation: We adopt the standard notations for regular surfaces as in \cite{Pogorelov, DoCarmo}. For $r(u, v)$ ($V(x, y)$ or $X(u, v))$ a parametrized surface, $E, F, G$ are data of the first fundamental form $ds^2=Edu^2+2F dudv +G dv^2$  (or $Edx^2+2Fdx dy+G dy^2$) and $L, M, N$ are the data of the second fundamental form. 

\section{First example}
Let
\[
  V(x,y) = \bigl(x,y,x^4-y^4\bigr), (x, y)\in \mathbb{R}^2
\]
be the parametrization of the graph of function $f(x, y)=x^4-y^4$. Denote this surface by $\Sigma$.
Then
\[
  \frac{\partial V}{\partial x}
    = \bigl(1,0,4x^3\bigr),
  \qquad
  \frac{\partial V}{\partial y}
    = \bigl(0,1,-4y^3\bigr).
\]
The unit normal of $\Sigma$  is
\[
  n
  =
  \frac{\bigl(-4x^3,4y^3,1\bigr)}
       {\sqrt{16x^6+16y^6+1}} .
\]
Hence the first fundamental form expressed in terms of the matrix is
\[
  I
  =
  \begin{pmatrix}
    1+16x^6 & -16x^3y^3 \\
    -16x^3y^3 & 1+16y^6
  \end{pmatrix}.
\]
Furthermore,
\[
  \frac{\partial^2 V}{\partial x^2}
    = \bigl(0,0,12x^2\bigr),
  \qquad
  \frac{\partial^2 V}{\partial y^2}
    = \bigl(0,0,-12y^2\bigr),
  \qquad
  \frac{\partial^2 V}{\partial x\,\partial y}
    = \bigl(0,0,0\bigr).
\]
Thus
\[
  L
  =
  \inner{\frac{\partial^2 V}{\partial x^2}}{n}
  =
  \frac{12x^2}{\sqrt{16x^6+16y^6+1}},
\]
\[
  M
  =
  \inner{\frac{\partial^2 V}{\partial x\,\partial y}}{n}
  =
  0,
\]
and
\[
  N
  =
  \inner{\frac{\partial^2 V}{\partial y^2}}{n}
  =
  \frac{-12y^2}{\sqrt{16x^6+16y^6+1}} .
\]
Therefore
\[
  II
  =
  \begin{pmatrix}
    \dfrac{12x^2}{\sqrt{16x^6+16y^6+1}} & 0 \\
    0 & \dfrac{-12y^2}{\sqrt{16x^6+16y^6+1}}
  \end{pmatrix}.
\]
It follows that the Gaussian curvature is
\[
  K
  =
  \frac{-144x^2y^2}
       {\bigl(16x^6+16y^6+1\bigr)^2}.
\]
The asymptotic curve equation is
\[
  L(\dd x)^2 + 2M\dd x\,\dd y + N(\dd y)^2 = 0.
\]
Equivalently,
\[
  x^2(\dd x)^2 - y^2(\dd y)^2 = 0.
\]

Now let
\[
  \gamma_a(t)
  =
  \bigl(a\cos t,a\sin t, a^4(\cos^4 t-\sin^4 t)\bigr), \quad a>0, 0\le t< 2\pi.
\]
Then
\[
  \dd x = -a \sin t\,\dd t,
  \qquad
  \dd y = a \cos t\,\dd t.
\]
Therefore
\[
  x^2(\dd x)^2-y^2(\dd y)^2
  =
 a^4( \cos^2 t \sin^2 t\,(\dd t)^2
  -
  \sin^2 t \cos^2 t\,(\dd t)^2)
  =
  0.
\]

Now we have constructed a family of closed asymptotic curves in $\Sigma$. Since it is a graph over the whole plane  $\Sigma$ is topologically $\mathbb{R}^2$.  Clearly the enclosed region of the surface is the graph restricted on the disc of radius $a$. The Gauss curvature calculation shows that it is not flat. The Gauss curvature of $\Sigma$ is negative except on the coordinate axes $x=0$ or $y=0$.

\begin{theorem}\label{thm:1} There exists a smooth saddle surface which is diffeomorphic to the plane with $K\le 0$ such that there are infinitely many closed asymptotic curves each enclosing a non-planar region.
\end{theorem}

The above construction can be put in a more general formulation as below.

For a function graph
  $V(x,y)=\bigl(x,y,f(x,y)\bigr)$
over the $(x,y)$-plane, let
\[
  D=1+\Big(\frac{\partial{f}}{\partial{x}}\Big)^2+\Big(\frac{\partial{f}}{\partial{y}}\Big)^2,
\]
then the determinant of the first fundamental form is $EG-F^2=D$, and

\[
  L=\frac{1}{\sqrt D}\frac{\partial^2 f}{\partial x^2} ,\qquad
  M=\frac{1}{\sqrt D}\frac{\partial^2 f}{\partial x\partial y} ,\qquad
  N=\frac{1}{\sqrt D}\frac{\partial^2 f}{\partial y^2} ,
\]
hence
\[
  K=\frac{LN-M^2}{EG-F^2}=\frac{f_{xx}f_{yy}-f_{xy}^2}{D^2}.
\]
In particular, the sign of $K$ coincides with the sign of the determinant of the
Hess$(f)$.

We now write $f$ in polar coordinates, i.e. $(x,y)=(r\cos\theta,r\sin\theta)$,
and consider a homogeneous function
\[
  f(r,\theta)=r^m\phi(\theta),
\]
where $m>1$ and $\phi$ is a smooth $2\pi$-periodic function. Under the
coordinate change from $(x,y)$ to $(r,\theta)$, $D=1+f_r^2+\frac{1}{r^2}f_\theta^2$  and the  quadratic form defined by the Hessian becomes
\[
\begin{aligned}
  f_{xx}(\dd x)^2+2f_{xy}\dd x\dd y+f_{yy}(\dd y)^2
  &=
  f_{rr}(\dd r)^2
  +2\left(f_{r\theta}-\frac1r f_\theta\right)\dd r\dd\theta 
  +\left(f_{\theta\theta}+r f_r\right)(\dd\theta)^2.
\end{aligned}
\]
Thus in polar coordinates 
$(r,\theta)$, the second fundamental form matrix is
\[
  \II
  = \big(1+f_r^2+\frac{1}{r^2}f_\theta^2\big)^{-\frac12}
  \begin{pmatrix}
    f_{rr} & f_{r\theta}-\frac1r f_\theta \\
    f_{r\theta}-\frac1r f_\theta & f_{\theta\theta}+r f_r
  \end{pmatrix}.
\]
For $f(r,\theta)=r^m\phi(\theta)$ this gives
\[
  \II
  =\big(1+r^{2m-2}\left(m^2\phi(\theta)^2+\bigl(\phi'(\theta)\bigr)^2\right)\big)^{-\frac12}
  \begin{pmatrix}
    m(m-1)r^{m-2}\phi(\theta)
    &
    (m-1)r^{m-1}\phi'(\theta)
    \\
    (m-1)r^{m-1}\phi'(\theta)
    &
    r^m\bigl(\phi''(\theta)+m\phi(\theta)\bigr)
  \end{pmatrix}.
\]
Therefore the Gauss curvature of the graph is
\[
\begin{aligned}
  K(r,\theta)
  &=
  \frac{(m-1)r^{2m-4}
  \left[
    m\phi(\theta)\bigl(\phi''(\theta)+m\phi(\theta)\bigr)
    -(m-1)\bigl(\phi'(\theta)\bigr)^2
  \right]}
  {
  \left[
    1+r^{2m-2}\left(m^2\phi(\theta)^2+\bigl(\phi'(\theta)\bigr)^2\right)
  \right]^2
  }.
\end{aligned}
\]

Now we require the closed circle $\{r=c\}$ to be asymptotic.
The tangent direction of such a circle is $\frac{\partial}{\partial\theta}$,
so the asymptotic condition is, in view of the second fundamental form:
\[
  \II\left(\frac{\partial}{\partial\theta},
           \frac{\partial}{\partial\theta}\right)=0,
\]
or equivalently, $f_{\theta\theta}+r f_r=0$.
For $f(r,\theta)=r^m\phi(\theta)$, this becomes $\phi''(\theta)+m\phi(\theta)=0$. In this case,
\[
  K(r,\theta)
  =
  -
  \frac{(m-1)^2r^{2m-4}\bigl(\phi'(\theta)\bigr)^2}
  {
  \left[
    1+r^{2m-2}\left(m^2\phi(\theta)^2+\bigl(\phi'(\theta)\bigr)^2\right)
  \right]^2
  }
  \leq 0.
\]
Consequently, the graph has nonpositive Gaussian curvature.

To obtain a globally well-defined smooth function on the plane, take $m=k^2, k\in\mathbb Z, k\geq 2$, and choose $\phi(\theta)=\cos(k\theta)$. Although the general solution of $\phi''+k^2\phi=0$ is
\[
  A\cos(k\theta)+B\sin(k\theta),
\]
it suffices to consider the cosine case, after rotating the $(x,y)$-axes.

Thus the functions
\[
  f_k(r,\theta)=r^{k^2}\cos(k\theta), k\in\mathbb Z, k\geq 2, 
\]
provide a family of saddle graphs. Returning to Cartesian coordinates,
we obtain
\[
  f_k(x,y)
  =
  (x^2+y^2)^{\frac{k^2-k}{2}}\,
  \operatorname{Re}\bigl((x+iy)^k\bigr).
\]
This is a polynomial, since $k^2-k$ is even. In particular, when $k=2$,
\[
  f_2(x,y)
  =
  (x^2+y^2)\operatorname{Re}\bigl((x+iy)^2\bigr)
  =
  (x^2+y^2)(x^2-y^2)
  =
  x^4-y^4
\]
gives the previous explicit counterexample to Pogorelov's question.

\section{The proof of Proposition \ref{prop:1} and outline of the second construction}

To prove the proposition we need the following well-known result related to Brouwer's fixed point theorem. 

\begin{lemma}\label{lemma:1} Let $D$ be a disk in $\mathbb{R}^n$. Then there is no continuous map
$$
w: D \to \partial{D} 
$$
such that $w(x)=x$ on $\partial D$.
\end{lemma}
See for example the proof of Theorem 3 of Section 8.1 of \cite{Evans}.

\begin{lemma}\label{lemma:2} If $V(x)$ is a continuous vector field defined in a neighborhood of $D$. If $V$ is not tangential to $\partial D$ for any $x\in \partial D$ $($in particular $V(x)\ne 0$ for $x\in \partial D$$)$, then $V$ must have a zero inside $D$.
\end{lemma}
\begin{proof} Without the loss of generality we assume $D$ is the unit ball.
Assume that $V$ has no zero inside $D$. Then $V(x)\ne 0$ for $x\in \overline{D}$. Now consider $W_0(x)=V(x)/|V(x)|$. Then $W_0: D \to \partial D$. The assumption implies, without the loss of generality, that
$$
\langle W_0(x), x\rangle > 0, \forall x\in \partial D.
$$
Now one can construct the continuous deformation
$$
w(x,t)=\frac{t\rho(x)x+(1-t)W_0(x)}{|t\rho(x)x+(1-t) W_0(x)|} 
$$
with $\rho(x)=1$ on $\partial D$ and decreases to zero as $x$ is away from the boundary to obtain a continuous map $W$ satisfying Lemma \ref{lemma:1}. 
\end{proof}

Now we can prove the proposition. Assume that there exists a closed asymptotic curve $\gamma$, then it encloses a disc $D$ by Jordan's theorem. Along the curve $\gamma(t)$ one may find the second asymptotic unit vector  $V(x)$ which is transversal to $\gamma(t)$, namely $V(x)$ never parallel to $\gamma'(t)$. This is contradictory with Lemma \ref{lemma:2}. To be precise we can construct $V$ as below.
Fix a point $o$ inside the disk. For a local coordinate near $o$ we may choose a local coordinate $(u, v)$ so that $r_u$ and $r_v$ are the principal directions with the principal curvature  $k_1>0>k_2$. For $e_1=\frac{1}{\sqrt{E}}r_u$ and $e_2=\frac{1}{\sqrt{G}}r_v$, if $e_\theta=e_1\cos \theta+e_2 \sin \theta$, the normal curvature of it is 
$$
k_\theta=\cos^2 \theta k_1 +\sin^2 \theta k_2
$$
by Euler's formula. Hence for $\theta_0=\arctan(\sqrt{\frac{-k_1}{k_2}})$ with $0< \theta_0<\frac{\pi}{2}$, $e_{\theta_0}$ is a unit asymptotic direction. The second one is $e_{\pi-\theta_0}$. Let $p$ be any other point, choose a path $\tau(t)$ joining $o$ to $p$. Repeating the construction along the path gives a pair of smooth unit asymptotic vectors and a pair of unit principal vectors along $\tau$ in each a local neighborhoods intersecting  with $\tau(t)$, and finally at $p$. Given that there exist two distinct asymptotic directions at each point,   and two principal directions (one corresponds to the positive principal curvature, the other to the negative), the simply-connectedness of $\Sigma$ ensures that there exists a pair of well-chosen smooth unit asymptotic vectors $V_1(t)$ and $V_2(t)$ at each point (as well as a pair of well-chosen unit principal vectors $V_3, V_4$). Since $\gamma'(t)$ is a smooth asymptotic vector along $\gamma$,  we have that $\gamma'(t)$ is parallel with one of $V_i$. Assume that $\gamma'(t)=|\gamma'(t)|V_1(t)$. Then $V_2(t)$ is not parallel with $\gamma'(t)$ along $\gamma$. We may choose $V(x)$ to be $V_2$.
  One can also choose $V$ to be one of $V_3, V_4$.

An obstruction of the above argument is that at the point where $K=0$, one principal direction coincides with the asymptotic direction and the obstruction arises if there is a planar umbilical point where no definite way to define the smooth $V(x)$ as the previous example shows. However, the proof works if there is no umbilical point on the surface even when $K\le 0$ since the assumption allows one to choose a smooth unit principal vector which transversal to the closed asymptotic curve.

\begin{corollary}\label{coro:1}
On a regular surface with  Gauss curvature $K\le 0$, which is homeomorphic to a plane, if we assume that there is no umbilical point, then there are no closed asymptotic curves.
\end{corollary}

 In the proof it is clear that one needs the simply-connectedness to construct a well-defined smooth unit vector transversal to $\gamma(t)$. Hence the second example will have nontrivial fundamental group.

  The construction starts from a closed space curve with non-vanishing curvature and non-vanishing torsion.  Taking the ruled surface generated by the principal normal lines of this curve produces a smooth immersion
\[
    X:S^1_{L_0}\times \mathbb R\longrightarrow \mathbb R^3,
    \qquad
    X(s,u)=\gamma(s)+uN_0(s),
\]
where $s$ is the arc-length parameter, $N_0$ is the principal normal of the curve, and $L_0$ is the length of the closed curve.  We verify directly that the induced metric is complete, that the Gauss curvature is strictly negative everywhere, and that the central curve $X(s,0)=\gamma(s)$ is a closed asymptotic curve.  Thus completeness plus the negativity of Gauss curvature  do not rule out closed asymptotic curves on saddle surfaces.

Here an asymptotic curve means a regular curve on the surface whose tangent direction has zero normal curvature.  Equivalently, if $\alpha$ is a curve on the surface, then $\alpha$ is asymptotic when
\[
    \II(\dot\alpha,\dot\alpha)=0
\]
along the curve.
We prove the following stronger counterexample to Toponogov's statement.

\begin{theorem}\label{thm:2}
There exists a complete smooth immersion
\[
    X:S^1_{L_0}\times \mathbb R\longrightarrow \mathbb R^3
\]
whose Gauss curvature satisfies $K<0$ everywhere and whose image contains a closed asymptotic curve.
\end{theorem}

It remains unclear if an embedded example exists. Below we include the details of the example.

\section{A closed space curve with non-vanishing torsion}

Consider the closed space curve
\[
    \Gamma(t)=\bigl((5+\cos 3t)\cos t,\ (5+\cos 3t)\sin t,\ \sin 3t\bigr),
    \qquad 0\leq t\leq 2\pi .
\]
The curve is smooth and closed.  It is also regular.  Indeed, its projection onto the $xy$-plane has polar form
\[
    r(t)=5+\cos 3t,
    \qquad \theta(t)=t,
\]
and $r(t)\geq 4$.  Hence the projected velocity has squared norm
\[
    (r'(t))^2+r(t)^2>0,
\]
so $\Gamma'(t)\neq 0$.

A direct computation gives
\[
\det\bigl(\Gamma'(t),\Gamma''(t),\Gamma'''(t)\bigr)
=
-3\bigl(8\cos^3(3t)-190\cos^2(3t)+2\cos(3t)+495\bigr).
\]
Let
\[
    c=\cos 3t,
    \qquad -1\leq c\leq 1.
\]
Then
\[
    8c^3+2c\geq -10,
    \qquad
    -190c^2\geq -190.
\]
Therefore
\[
    8c^3-190c^2+2c+495
    \geq 495-190-10
    =295>0.
\]
It follows that
\[
    \det\bigl(\Gamma',\Gamma'',\Gamma'''\bigr)<0
\]
for every $t$.  Recall that for curve $\Gamma(t)$ the curvature and torsion are given by the formulae (cf. \cite{Pogorelov})
$$
\kappa(t)=\frac{|\Gamma'\times \Gamma''|}{|\Gamma'|^3}, \quad \tau(t)=\frac{\det(\Gamma', \Gamma'', \Gamma''')}{|\Gamma'\times\Gamma''|^2}.
$$
In particular, the curvature and the torsion of $\Gamma$ are both nonzero everywhere.

Now reparametrize $\Gamma$ by arclength.  Denote the resulting curve by
\[
    \gamma:S^1_{L_0}\longrightarrow \mathbb R^3,
\]
where $L_0$ is the total length of $\Gamma$ and
\[
    S^1_{L_0}=\mathbb R/L_0\mathbb Z.
\]
Let
\[
    T(s),\quad N_0(s),\quad B(s)
\]
be the Frenet frame of $\gamma$, and let $\kappa(s)$ and $\tau(s)$ be its curvature and torsion.  Then
\[
    \kappa(s)>0,
    \qquad
    \tau(s)\neq 0
\]
for all $s\in S^1_{L_0}$.  The Frenet equations are
\[
    T'=\kappa N_0,
    \qquad
    N_0'=-\kappa T+\tau B,
    \qquad
    B'=-\tau N_0.
\]
Here we use the convention of \cite{DoCarmo}. 
Because $S^1_{L_0}$ is compact and $\tau$ is continuous and nowhere zero, there exists a constant $\delta>0$ such that
\[
    |\tau(s)|\geq \delta>0
\]
for all $s\in S^1_{L_0}$.

\section{Construction of the surface}

Define
\[
    X:S^1_{L_0}\times \mathbb R\longrightarrow \mathbb R^3,
    \qquad
    X(s,u)=\gamma(s)+uN_0(s).
\]
The parameter $u$ is allowed to range over the whole real line.  Therefore this is not merely a narrow local strip around $\gamma$; it is a surface defined on the whole cylinder $S^1_{L_0}\times \mathbb R$.

Using the Frenet formulas, we obtain
\[
    X_s=T+uN_0'=(1-\kappa u)T+\tau uB,
    \qquad
    X_u=N_0.
\]
Hence the coefficients of the first fundamental form are
\[
    E=\langle X_s,X_s\rangle=(1-\kappa u)^2+\tau^2u^2,
    \qquad
    F=\langle X_s,X_u\rangle=0,
    \qquad
    G=\langle X_u,X_u\rangle=1.
\]
Since $\tau(s)\neq 0$, we have
\[
    E=(1-\kappa u)^2+\tau^2u^2>0
\]
for every $(s,u)\in S^1_{L_0}\times \mathbb R$.  Thus $X_s$ and $X_u$ are linearly independent everywhere, and $X$ is a regular immersion.

\section{Completeness of the induced metric}

The induced metric is
\[
    I=E\,ds^2+du^2.
\]
For fixed $s$, the function
\[
    E(s,u)=(1-\kappa(s)u)^2+\tau(s)^2u^2
\]
is a positive quadratic polynomial in $u$.  Its minimum is
\[
    \min_{u\in\mathbb R} E(s,u)
    =
    \frac{\tau(s)^2}{\kappa(s)^2+\tau(s)^2}.
\]
Since $s\in S^1_{L_0}$ and $\kappa,\tau$ are continuous with $\tau$ nowhere zero, the above positive function has a positive minimum on $S^1_{L_0}$.  Hence there exists $m>0$ such that
\[
    E(s,u)\geq m>0
\]
for all $(s,u)\in S^1_{L_0}\times \mathbb R$.

Consequently,
\[
    I=E\,ds^2+du^2
    \geq
    m\,ds^2+du^2.
\]
The metric $m\,ds^2+du^2$ is the standard complete product metric on the cylinder $S^1_{L_0}\times \mathbb R$, up to a constant factor in the circle direction.  Therefore every divergent curve in $S^1_{L_0}\times \mathbb R$ has infinite length with respect to $m\,ds^2+du^2$, and hence also has infinite length with respect to $I$.  Thus the induced metric $I$ is complete.

\section{Computation of the Gauss curvature}

A unit normal field is
\[
    n
    =
    \frac{(1-\kappa u)B-\tau uT}
    {\sqrt{(1-\kappa u)^2+\tau^2u^2}}.
\]
Since
\[
    X_{uu}=0,
\]
the coefficient $N$ of the second fundamental form is
\[
    N=\langle X_{uu},n\rangle=0.
\]
Moreover,
\[
    X_{su}=N_0'=-\kappa T+\tau B.
\]
Therefore the mixed coefficient $M$ is
\[
\begin{aligned}
    M
    &=\langle X_{su},n \rangle  \\
    &=
    \left\langle -\kappa T+\tau B,
    \frac{(1-\kappa u)B-\tau uT}
    {\sqrt{(1-\kappa u)^2+\tau^2u^2}}
    \right\rangle \\
    &=
    \frac{\tau}{\sqrt{(1-\kappa u)^2+\tau^2u^2}}.
\end{aligned}
\]
Since $F=0$ and $G=1$, the Gauss curvature is
\[
\begin{aligned}
    K
    &=\frac{LN-M^2}{EG-F^2} \\
    &=\frac{0-M^2}{E} \\
    &=-\frac{\tau^2}{E^2} \\
    &=-\frac{\tau^2}{\bigl((1-\kappa u)^2+\tau^2u^2\bigr)^2}.
\end{aligned}
\]
Since $\tau(s)\neq 0$ for all $s$, we conclude that
\[
    K<0
\]
everywhere on $S^1_{L_0}\times \mathbb R$.

\section{The closed asymptotic curve}

Consider the central curve on the surface:
\[
    \alpha(s)=X(s,0)=\gamma(s).
\]
Because $s$ is a periodic arclength parameter on $S^1_{L_0}$, the curve $\alpha$ is closed.

At $u=0$, we have
\[
    X_s(s,0)=T(s),
    \qquad
    n(s,0)=B(s).
\]
Also,
\[
    X_{ss}(s,0)=T'(s)=\kappa(s)N_0(s).
\]
Therefore
\[
\begin{aligned}
    \II(T,T)
    &=\langle X_{ss}(s,0),n (s,0)\rangle \\
    &=\langle \kappa(s)N_0(s),B(s)\rangle \\
    &=0.
\end{aligned}
\]
Hence the tangent direction of $\alpha$ is asymptotic at every point.  Thus
\[
    \alpha(s)=X(s,0)=\gamma(s)
\]
is a closed asymptotic curve on the surface.

In summary we have constructed a complete smooth immersion
\[
    X:S^1_{L_0}\times \mathbb R\longrightarrow \mathbb R^3,
    \qquad
    X(s,u)=\gamma(s)+uN_0(s),
\]
whose Gauss curvature satisfies
\[
    K
    =
    -\frac{\tau^2}
    {\bigl((1-\kappa u)^2+\tau^2u^2\bigr)^2}
    <0
\]
everywhere, and which contains the closed asymptotic curve
\[
    X(s,0)=\gamma(s).
\]
Therefore Toponogov's statement is false in the immersed-surface category.  Even after adding completeness, negative Gauss curvature alone does not exclude closed asymptotic curves.

\section*{Appendix: a clarification of a conjecture of Milnor stated in \cite{Topo}} 
 Let \(\Sigma\subset \mathbb{R}^3\) be a connected complete surface. Denote by \(k_1\) and \(k_2\) its principal curvatures, and by
\(K=k_1k_2\)
its Gaussian curvature. The surface is called saddle if $K\le 0$. On page 113 of \cite{Topo} it states:

{\it In 1966, J. Milnor enunciated a conjecture, which for a saddle surface implies that for problem 2.6.10 to hold, it suffices that:

\[\inf_{P\in \Sigma}(|k_1|(P)+|k_2|(P))\ne 0\]
This assertion has not yet been proven.}

The Problem  2.6.10 on page 113 of \cite{Topo} states as

{ \it Prove that if on a complete saddle surface $\Sigma\subset \mathbb{R}^3$, 

\[\inf_{P\in \Sigma}|k_1|(P)+\inf_{P\in \Sigma}|k_2|(P)> 0\]
holds, then $\Sigma$ is a cylinder, i.e. $k_1(P)\equiv 0$ and $k_2(P)=const$.}

On the other hand,  the proof presented for this statement on page 113 of \cite{Topo} only shows that $\Sigma$ is a general cylinder.

Hence the conclusion in the statement of Milnor's conjecture as well as in Problem 2.6.10 should be {\it 
$K\equiv 0$, namely $\Sigma$ is a general cylinder.}

In other words, `being a cylinder' should not be interpreted as a round cylinder insisting that the nonzero principal curvature is constant. The example below illustrates  this point.

\begin{proposition}
There exists a surface \(\Sigma\subset\mathbb R^3\) such that
\[ \Sigma\cong\mathbb R^2,\qquad
 \Sigma \text{ is complete and embedded},\qquad
 K\equiv0,
 \qquad\inf_{P\in \Sigma}(|k_1|(P)+|k_2|(P))>0,\]
and such that one principal curvature is identically zero while the other principal curvature is not constant. Moreover, the same example satisfies
\[\inf_{P\in \Sigma}|k_1|(P)+\inf_{P\in \Sigma}|k_2|(P)> 0\]
\end{proposition}

\begin{proof}
Define
\[ \rho(t)=2+\frac12\tanh t\]
and let
\[\gamma(t)=(\rho(t)\cos t,\rho(t)\sin t),\qquad t\in\mathbb R.\]
Consider the generalized cylinder
\[ X(t,v)=\bigl(\rho(t)\cos t,\,\rho(t)\sin t,\,v\bigr), \qquad (t,v)\in\mathbb R^2,\]
obtained by translating the plane curve \(\gamma\) in the \(z\)-direction. We set \(\Sigma=X(\mathbb R^2)\).

We first verify that \(\Sigma\) is an embedded regular surface homeomorphic to \(\mathbb R^2\). Since
\[
 \rho'(t)=\frac12\sech^2t>0,
\]
the function \(\rho\) is strictly increasing, and \(\rho(t)>0\) for all \(t\). If \(\gamma(t_1)=\gamma(t_2)\), then the polar angles satisfy \(t_1\equiv t_2\pmod{2\pi}\), while the radii satisfy \(\rho(t_1)=\rho(t_2)\). The strict monotonicity of \(\rho\) gives \(t_1=t_2\). Thus \(\gamma\) has no self-intersections. It follows that
\[
 X(t,v)=(\gamma(t),v)
\]
is injective. Moreover,
\[X_t=\bigl(\rho'\cos t-\rho\sin t,\,\rho'\sin t+\rho\cos t,\,0\bigr),
 \qquad
 X_v=(0,0,1).\]
The vector \(X_t\) is nonzero because
\[
 |X_t|^2=\rho^2+(\rho')^2>0,
\]
and \(X_v\) is independent of \(X_t\). Hence \(X\) is an immersion. Finally, convergence in the image determines convergence of the radius and hence of the parameter \(t\), while the third coordinate determines \(v\). Therefore \(X\) is an embedding and \(\Sigma\) is homeomorphic to \(\mathbb R^2\).

Next we compute the first fundamental form and prove completeness. From the above formulas for \(X_t\) and \(X_v\),
\[ E=\langle X_t,X_t\rangle=\rho^2+(\rho')^2,
 \qquad
 F=\langle X_t,X_v\rangle=0,
 \qquad
 G=\langle X_v,X_v\rangle=1.\]
Therefore
\[ I=(\rho^2+(\rho')^2)\,dt^2+dv^2.\]
Since \(\tanh t\in(-1,1)\), one has
\[
 \frac32<\rho(t)<\frac52.
\]
In particular,
\[\rho^2+(\rho')^2\ge \rho^2>\frac94,\]
and hence
\[I\ge \frac94\,dt^2+dv^2.\]
The metric on the right is a complete flat metric on \(\mathbb R^2\). Thus every divergent curve in the parameter plane has infinite length with respect to \(I\), and \(\Sigma\) is complete.

We now compute the principal curvatures. For a generalized cylinder \(X(t,v)=(\gamma(t),v)\), the \(v\)-curves are straight lines. Hence the corresponding principal curvature is
\[ k_2=0.\]
The other principal curvature is the curvature of the plane curve \(\gamma\). Since \(\gamma(t)=\rho(t)(\cos t,\sin t)\), the polar-coordinate curvature formula gives
\[ k_1(t)=\frac{\rho^2+2(\rho')^2-\rho\rho''}{\bigl(\rho^2+(\rho')^2\bigr)^{3/2}}.\]
Here
\[ \rho'(t)=\frac12\sech^2t,
 \qquad
 \rho''(t)=-\tanh t\,\sech^2t.\]
Consequently
\[ K=k_1k_2\equiv0.\]
In particular, \(\Sigma\) is flat and satisfies \(K\le0\).

It remains to prove that the principal curvature size is bounded away from zero. Since \(k_2=0\), it suffices to show
\[ \inf_{t\in\mathbb R}|k_1(t)|>0.\]
Let
\[ A(t)=\rho^2+2(\rho')^2-\rho\rho''.\]
Since
\[ |\rho''(t)|=|\tanh t|\sech^2t,\]
putting \(q=|\tanh t|\in[0,1]\) gives
\[ |\rho''(t)|=q(1-q^2).\]
The maximum of \(q(1-q^2)\) on \([0,1]\) is
\[ \frac{2}{3\sqrt3}.\]
Using \(3/2<\rho(t)<5/2\), we obtain
\[ A(t)=\rho^2+2(\rho')^2-\rho\rho''
 \ge \rho^2-\rho|\rho''|
 \ge \frac94-\frac52\cdot\frac{2}{3\sqrt3}>0.\]
On the other hand,
\[ \rho^2+(\rho')^2\le \left(\frac52\right)^2+\left(\frac12\right)^2=\frac{13}{2}.\]
Therefore
\[\bigl(\rho^2+(\rho')^2\bigr)^{3/2}
 \le \left(\frac{13}{2}\right)^{3/2}.\]
Combining the two estimates yields
\[k_1(t)
 \ge
 \frac{\displaystyle \frac94-\frac52\cdot\frac{2}{3\sqrt3}}
 {\displaystyle \left(\frac{13}{2}\right)^{3/2}}>0.\]
Hence
\[
 \inf_{\Sigma}(|k_1|+|k_2|)
 =\inf_{t\in\mathbb R}|k_1(t)|>0.
\]
In particular, \(\Sigma\) has no umbilic points, since its principal curvatures are \(k_1>0\) and \(k_2=0\).

Finally we show that \(k_1\) is not constant. As \(t\to +\infty\),
\[\rho(t)\to\frac52,
 \qquad
\rho'(t)\to0,
\qquad
\rho''(t)\to0.\]
Therefore
\[\lim_{t\to+\infty}k_1(t)
=\frac{(5/2)^2}{\bigl((5/2)^2\bigr)^{3/2}}
=\frac1{5/2}
=\frac25.\]
As \(t\to -\infty\),
\[ \rho(t)\to\frac32,
 \qquad
 \rho'(t)\to0,
 \qquad
 \rho''(t)\to0.\]
Thus
\[ \lim_{t\to-\infty}k_1(t)
=\frac{(3/2)^2}{\bigl((3/2)^2\bigr)^{3/2}}
=\frac1{3/2}
=\frac23.\]
Since \(2/5\ne2/3\), the function \(k_1\) is not constant. We have therefore constructed a complete embedded surface \(\Sigma\cong\mathbb R^2\) satisfying \(K\equiv0\) and \(\inf_\Sigma(|k_1|+|k_2|)>0\), with one principal curvature identically zero and the other principal curvature nonconstant.
\end{proof}

\bigskip

\noindent\textbf{Acknowledgments.} {The authors thank ChatGPT for helping to construct these examples. The questions studied were noticed by the first author when he taught  Math 150 B at UCSD in the Winter quarter of 2024, and the Honors Differential Geometry class at ZJNU in Spring of 2026. He would like to thank the students from both classes for their participation.}

\end{document}